\documentclass[a4paper,11pt]{amsart}
\usepackage[left=2.7cm,right=2.7cm,top=3.5cm,bottom=3cm]{geometry}
 \usepackage[colorlinks]{hyperref}
\usepackage{amsthm,amssymb,amsmath,amsfonts,amscd}
\usepackage{stmaryrd}
\usepackage[latin1]{inputenc}
\usepackage{graphicx}

 \usepackage{latexsym}
 \usepackage{longtable,bbold}
 \usepackage[all]{xy}

\usepackage{tikz}
\usetikzlibrary{matrix,arrows}
\usepackage{mathabx,epsfig}
\newfont{\cyr}{wncyr10 scaled 1100}
\newfont{\cyrr}{wncyr9 scaled 1000}

\catcode`\=\active\def {\'e}
\catcode`\=\active\def {\`e}
\catcode`\=\active\def {\`a}
\catcode`\=\active\def {\^a}
\catcode`\Ë=\active\def Ë{\`A}
\catcode`\=\active\def {\"a}
\catcode`\=\active\def {\`u}
\catcode`\=\active\def {\^{e}}
\catcode`\=\active\def {\^{i}}
\catcode`\=\active\def {\"i}
\catcode`\=\active\def {\^{o}}
\catcode`\=\active\def {\^{u}}
\catcode`\=\active\def {\"{u}}
\catcode`\=\active\def {\"{U}}
\catcode`\=\active\def {\c{c}}
\catcode`\=\active\def {\c{C}}
\catcode`\é=\active\def é{\c{C}}

\theoremstyle{plain}
\newtheorem{theorem}{Theorem}[section]

\newtheorem*{teono}{Theorem}

\theoremstyle{remark}
\newtheorem{remark}[theorem]{Remark}
\newtheorem{question}[theorem]{Question}

\newtheorem{thm}{Theorem}[section]
\newtheorem{prop}[thm]{Proposition}
\newtheorem{lem}[thm]{Lemma}
\newtheorem{cor}[thm]{Corollary}

\newtheorem{que}[thm]{Question}
\newtheorem{rmk}[thm]{Remark}
\newtheorem{dfn}[thm]{Definition}

\newtheorem{ass}[thm]{Hypothesis}

\newcommand{\GL}{{\operatorname{GL}}}

\newcommand{\ha}{{\operatorname{ha}}}

\newcommand{\gerp}{{\mathfrak{p}}}
\newcommand{\gerq}{{\mathfrak{q}}}

\def\N{\mathbb{N}}

\newcommand{\set}[1]{\left\lbrace #1 \right\rbrace}
\newcommand{\mb}[1]{\mathbb{#1}}
\newcommand{\mc}[1]{\mathcal{#1}}
\newcommand{\mr}[1]{\mathrm{#1}}
\newcommand{\mf}[1]{\mathfrak{#1}}

\begin{document}

\title{Perfectoid Drinfeld modular forms} 
\author{Marc-Hubert Nicole \& Giovanni Rosso}

 \begin{abstract} 
In the first part, we revisit the theory of Drinfeld modular curves and $\pi$-adic Drinfeld modular forms for $\GL(2)$ from the perfectoid point of view.
In the second part, we review open problems for families of Drinfeld modular forms for $\GL(N)$.
 
 \end{abstract}


\address{M.-H. N.: Aix Marseille Universit, CNRS, Centrale Marseille, I2M, UMR 7373, 13453 Marseille. Mailing Address: Universit d'Aix-Marseille, campus de Luminy, case 907, Institut mathmatique de Marseille (I2M), 13288 Marseille cedex 9, FRANCE}
\email{\href{mailto: marc-hubert.nicole@univ-amu.fr}{marc-hubert.nicole@univ-amu.fr}}

\address{G.R.: Concordia University, Department of Mathematics and Statistics,
Montral, Qubec, Canada}
\email{\href{mailto:giovanni.rosso@concordia.ca}{giovanni.rosso@concordia.ca}}
\urladdr{\url{https://sites.google.com/site/gvnros/}}

\date{version of \today}

\maketitle
 % \tableofcontents

\section{Introduction} 
 Let $\mc C$ be a projective smooth curve over $\mb F_q$, $F=\mb F_q (\mc C)$ its function field, $\infty$ an $\mb F_q$-rational point and $A$ the ring of regular functions outside $\infty$. Fix an $A$-ideal $\mf N$, that we can suppose principal and a prime ideal $\gerp$ of norm $q^d$, coprime with $\mf N$. Let $\pi$ be a generator of $\gerp$ in $A_{\gerp}$.
 
In our previous paper \cite{NicoleRosso}, we explained how to adapt the eigenvariety machinery to the non-noetherian context of Drinfeld modular varieties associated to $\GL(n)$ for $n \in \N$, including Hida theory in the form of an analogue of the Vertical Control Theorem, a continuous analogue of Coleman's finite slope families and a classicality theorem of overconvergent Drinfeld modular forms. This led to a variety of open questions, some intrinsic to the set-up of Drinfeld modular forms.\\

In the first part of this paper, we illustrate in detail that the perfectoid approach to Shimura varieties pioneered by Scholze \cite{ScholzeTorsion} also works well for Drinfeld modular curves associated to $\GL(2)$ (where there are no technical difficulties at the boundary, exactly as for classical modular curves), and allows us to give an alternative treatment of overconvergent Drinfeld modular forms following Chojecki--Hansen--Johansson \cite{CHJ}. 
After reviewing the theory of Drinfeld modules and the Hodge--Tate--Taguchi map, we first show that there exists an infinite level Drinfeld modular curve which is a perfectoid space; we follow closely the construction of Scholze, first constructing an anti-canonical tower of a strict neighbourhood of the ordinary locus and then using the Hodge--Tate--Taguchi map to extend it to the whole Drinfeld modular curve, thus proving 

\begin{teono}[Thm. \ref{thm:perfectoid}]
Let $\mc X(\pi^m)$ be the Drinfeld modular curve of level $\Gamma(\gerp^m)$. There exists a perfectoid space 
\[
\mc X_{\infty} \sim \varprojlim_m \mc X(\pi^m)^{\mr{perf}},
\]
where $\mc X(\pi^m)^{\mr{perf}}$ is the perfection of $\mc X(\pi^m)$. This space is equipped with a  $\mr{GL}_2(A_{\gerp})$-equivariant Hodge--Tate--Taguchi map of adic spaces 
\[
{\Pi_{\mr{HTT}}}: \mc X_{\infty} \rightarrow \mb P^1.
\]
\end{teono}

Then, given an analytic weight $s \in \mb Z_p$ we define following \cite{CHJ} overconvergent Drinfeld modular forms as function of the variable $\mf z $ (a so-called fake Hasse invariant, which is the pullback via the Hodge--Tate--Taguchi map of the coordinate $z$ on $\mb P^1$) which satisfy the transformation formula

\begin{equation}
\label{aboveeq}
f(\gamma \mf z )= {(b\mf z+d)}^{-s}f(\mf z) 
\end{equation}

\noindent
for all $\gamma \in \Gamma_0(\gerp)$. More precisely, we construct a line bundle  $\omega^s$ on $\mc X_0(\pi)(v)^{\mr{perf}}$ consisting of functions satisfying the above formula \eqref{aboveeq} and we show that this sheaf is the pullback from $\mc X_0(\pi)(v)$ of the sheaf of overconvergent Drinfeld modular forms of weight $s$ defined in \cite{NicoleRosso}.

In the second part of this note, we treat a variety of open problems of widely varying level of difficulty in some detail. In brief: a conjectural $r=t$ theorem; asking for a better definition of the Fredholm determinant in the non-noetherian context; asking about families of generalized modular forms for Anderson motives; the study of slopes  la Gouva-Mazur in higher rank; classicity in infinite slope - an example of a problem arising only for Drinfeld modular forms; establishing a variant of Hida's Horizontal Control theorem for $\GL(n)$.

 \subsubsection*{Acknowledgement} We thank Andrea Bandini, Christian Johansson, David Hansen, Vincent Pilloni, Roberto Svaldi, Maria Valentino for useful discussions during the preparation of this paper.
\section{Perfectoid Drinfeld modular curves}

  Let $\mc X=\mc X(\mf N)$ be the compactified Drinfeld modular curve of full level $\mf N$ seen as an adic space over $\mr{Spa}(F_{\gerp},A_{\gerp})$. 

The main theorem of the section is the following:

\begin{thm}
There exists a perfectoid space
\[
\mc X_{\infty} \sim \varprojlim_n \mc X(\pi^{n})^{\mr{perf}}
\]
equipped with a natural map $\mr{GL}_2(F_{\gerp})$-equivariant Hodge--Tate--Taguchi period map to $\mb P^1$. 
\end{thm}
\noindent
We need to consider the perfection of $ \mc X(\pi^{n})$ as in characteristic $p$ perfectoid algebras are necessarily perfect.

As in \cite{CHJ} we shall use this map to define some overconvergent $\pi$-adic modular forms of $p$-adic weight $s$ as functions on $\mb P^1$ satisfying the usual transformation property 
\[ 
f \vert_s \gamma (z) = j(\gamma,z)^s f(z),
\]
for a $\pi$-adic cocycle $ j(-,z)$.
The proof of the theorem follows the lines of \cite{ScholzeTorsion}: we first construct a perfectoid anti-canonical tower over a strict neighbourhood of the ordinary locus using the fact that we have a map 
\[
\mc X(q^{-dm}v) \rightarrow \mc X_0(\pi^{m}).
\]

\noindent
On points, the map sends a rank two Drinfeld module $\varphi$ to $(\varphi/C_m, \varphi[\pi^m]/C_m)$, for $C_m$ the canonical subgroup of level $m$. The remarkable feature of this map is that the Hasse invariant of $\varphi/C_m$ is the Hasse invariant of $\varphi$ multiplied by $q^d$. Hence, the overconvergence radius on the image is constant independent of $m$.

This allows us to construct an intermediate perfectoid object $\mc X_{0,\infty}(v)$ over $\mc X(v)$. Then we use the purity theorem to go from level $\Gamma_0$ to full level without much ado, as we are working over Drinfeld modular curves. In general the map from level $\Gamma_1$ to level $\Gamma_0$ is not \'etale on the boundary.

On $\mc X_{\infty}(v)$ we can the define the Hodge--Tate--Taguchi map to $\mb P^1$, and use it to show that, at the level of topological spaces, $\vert \mc X_{\infty} \vert$ is covered by finitely many translates of $\vert \mc X_{\infty}(v)\vert$. This gives a cover of $\vert \mc X_{\infty} \vert$ by perfectoid spaces, hence the results.

\begin{remark}
The same construction should work for general rank $r-1$, but studying the boundary becomes trickier. The Hodge--Tate--Taguchi map will take values in the flag variety parameterizing flags with blocks of size $r-1$ and $1$, which is isomorphic to $\mb P^{r-1}$. One can then define overconvergent Drinfeld modular forms as functions on $\mb P^{r-1}$ satisfying suitable transformation properties, exactly as in the analytic case \cite{BBPPartI}. 
\end{remark}

\subsection{Reminder on canonical subgroups and the Hodge--Tate--Taguchi decomposition}
Fix a  lift $\mr{Ha}$ of the Hasse invariant as in \cite[\S 4]{NicoleRosso} and let $\mr{ha}$ be the truncated valuation of the Hasse invariant. For $v \in \mb Q \cap [0,1]$ let $\mc X(v)$ be a strict neighbourhood of the ordinary locus of Drinfeld modules for which $\mr{ha} \leq v$.
Given a (formal) Drinfeld module $\varphi$, we can take its Taguchi dual $\varphi^{D}:=\mr{Hom}(
\varphi, \mr{CH})$, where CH denotes the Carlitz--Hayes module, the unique (formal) Drinfeld module of rank $1$ and good reduction at $\gerp$.

There is an Hodge--Tate--Taguchi map
\[
\mr{HTT}^{D}_{\varphi,m}:\varphi^{D}[\gerp^n](\overline{K}) \rightarrow \omega_{\varphi}/\pi^n \mc O_{\overline{K}}
\]
sending a torsion point $x_m \in \mr{Hom}(
\varphi, \mr{CH})$ to $x_m^* \textup{d}z$, for $\textup{d}z$ the canonical differential on CH.
We have a so-called dual version of it (without using a base of $\omega_{\mr{CH}}$):
\[
\mr{HTT}_{\varphi}:T_\gerp(\varphi) \rightarrow \mr{Lie}(\varphi^D)^\vee.
\]
Indeed, by definition of the Tate module, any $x \in T_\gerp(\varphi)$ can be seen as a map $F_{\gerp}/A_{\gerp} \rightarrow \varphi$. There is a dual map $x^D : \varphi^{D} \rightarrow \mr{CH}$, which defines a map 
 \[\mr{HTT}_{\varphi}(x)= \mr{Lie}(x^D) \in \mr{Lie}(\varphi^{D})^\vee.\]
 
We recall the following theorem \cite[Th\'eor\`eme 4.4]{NicoleRosso}

\begin{thm}\label{thm:sousgroupe}
Let $m \geq 1$ be a positive integer. 
Let $v \in \mb Q \cap [0,1]$  such that $ v< \frac{1}{2q^{d(m-1)}}$. 
\begin{itemize}
\item[(i)]
Over $\mc X(v)$, the $\mf p^m$-torsion of the generalised Drinfeld $(\overline{\mc E},\overline{\varphi})$ has a canonical subgroup $C_{\overline{\mc E},m}$ of $1$, dimension $1$ and level $m$;
\item[(ii)]  For all formal open $\mr{Spf}(R)$ of  $\mc X(v)$, the linearisation of the Hodge--Tate--Taguchi map
\[
\mr{HTT}:{C_{\overline{\mc E},m}^{D}}(\overline{R}) \otimes  \overline{R}/\pi^m \overline{R} \rightarrow \omega_{C_{\overline{\mc E},m}}\otimes  \overline{R}/\pi^m \overline{R} 
\]
 has cokernel killed by $\pi^w$, for all $w \in  \mb Q_{>0}$ such that $w \geq \frac{v}{q^d-1}$.
\end{itemize} 
\end{thm}

We now prove a useful lemma
\begin{lem}\label{lem:lubinkatz}
Let $x=({\mc E}_x,{\varphi}_x)$ be a Drinfeld module and suppose that $\ha(x)=v  \leq  \frac{1}{q^d+1} $. Then $\ha(({\mc E}_x/C_{\overline{\mc E},1},{\varphi}_x/C_{\overline{\mc E},1}))=q^d v$.
\end{lem}
\begin{proof}
As $   \frac{1}{q^d+1} \leq  \frac{1}{2}$ we have a canonical subgroup and the proof is, exactly as in the case of elliptic curves, a study of the Newton polygon, see \cite[Theorem 3.10.7]{Katz}.
\end{proof}

Let $({\mc E},{\varphi})$ be a Drinfeld module over an algebraically closed and complete field $C /F_{\gerp}$. We will define the Hodge--Tate--Taguchi decomposition of $T_\gerp(\varphi)$. 
\begin{thm}
\label{previousthm}
Let $\varphi$ be a Drinfeld module over $K$ of rank two, and let $\varphi^D$ its Taguchi dual. 
We have a surjective map 
\[
 T_\gerp(\varphi^D)\otimes_{A_{\gerp}} \overline{K}  \rightarrow \mr{Lie}(\varphi)^{\vee}\otimes_{K} \overline{K} 
\]
which induces a line $\mr{Lie}(\varphi_{\overline{K}})$ inside $T_\gerp(\varphi)\otimes_{A_{\gerp}} \overline{K} $.
\end{thm}
\begin{proof}
As $\mc X$ is proper, we can find a generalised Drinfeld module $\widetilde{\varphi}$ over $\mc O_K$ whose generic fiber is $\varphi$. 
\\

If $\widetilde{\varphi}$ corresponds to a point in $\mc X(v)$, then let $x$ be a lift to $ T_\gerp(\varphi^D)$ of a generator of its canonical subgroup $C_{\widetilde{\varphi},m}^{D}(\overline{K})$. By Theorem \ref{thm:sousgroupe} we know that, up to $\pi^{wn}$, $\mr{HTT}^D_{\widetilde{\varphi}}(x)$ generates $\omega_{C_{\widetilde{\varphi},m}}$, and hence over $K$ the Hodge--Tate--Taguchi map is surjective.
Note that $T_\gerp(\varphi^D)$ is the dual representation: given
$y\in T_\gerp(\varphi^D)=\mr{Hom}(F_{\gerp}/A_{\gerp}, \mr{Hom}(\varphi, \mr{CH}))$ and $x \in T_\gerp(\varphi)=\mr{Hom}(F_{\gerp}/A_{\gerp}, \varphi)$, we get $\langle x,y \rangle \in \mr{Hom}(F_{\gerp}/A_{\gerp}, \mr{CH})$ defined by
\[
\langle x,y\rangle (z) : = y(z)(x(z)).
\]
By duality, we get a line in $T_\gerp(\varphi)\otimes_{A_{\gerp}} \overline{K}$. Note that this line is the kernel of $\mr{HTT}_{\varphi} \otimes \mr{id}$.

If the point corresponding to $\widetilde{\varphi}$ is not in $\mc X(v)$ then it is not on the boundary as Tate--Drinfeld modules are ordinary, see e.g., \cite[Lemma 4.1]{HattoriKM}. By \cite[Lemme 5.2]{NicoleRosso} quotienting $\widetilde{\varphi}$ by a suitable subgroup of the $\gerp^m$-torsion will move the point into $\mc X(v)$, and the Tate modules of two isogenous Drinfeld modules are merely different lattices in $T_\gerp(\varphi)\otimes_{A_{\gerp}} F_{\gerp} $.
\end{proof}

\begin{rmk}
If $r>2$, the Hodge--Tate--Taguchi map is surjective only on the locus of good reduction. Indeed, if the reduction of the Drinfeld modules $\varphi$ is not good, the corresponding point in $\mc X(\mc O_C)$ will fall in the boundary, which for $r>2$ is not necessarily ordinary.
\end{rmk}

\begin{dfn}[Hodge--Tate--Taguchi period map]\label{def:HTT}
Given a trivialisation $\eta : T_\gerp(\varphi) \cong A^2_{\gerp}$, we can define an element ${\Pi_{\mr{HTT}}}(\varphi, \eta) \in  \mb P^{1}(\overline{K})$ by picking the line in ${\overline{K}}^2$ given by $\mr{Lie}(\varphi_{\overline{K}})$.
\end{dfn}
Before constructing the perfectoid tower, we prove an important property of the Hodge--Tate--Taguchi period map:
\begin{lem}\label{lem:P^1(F)}
Let $({\mc E},{\varphi})$ be a rank $2$ Drinfeld module over an algebraically closed field $K$, then $ {\Pi_{\mr{HTT}}}(\varphi, \eta) \in \mb P^{1}(F_{\gerp})$ if and only if $({\mc E},{\varphi})$ is $\gerp$-ordinary.
\end{lem}

\begin{rmk} This is the exact analogue of what happens for classical (perfectoid) modular curves mapping the supersingular locus to the Drinfeld upper-half plane via the Hodge-Tate map.
\end{rmk}

\begin{proof}
If $({\mc E},{\varphi})$ is ordinary, then we can identify $\mr{Lie}(\varphi)$ with the Lie algebra of the canonical subgroup (which is the Carlitz--Hayes module), and the its $\pi$-torsion is an $A_{\gerp}$-line in $T_\gerp(\varphi)$. 
Conversely, let $({\mc E},{\varphi})$ be a Drinfeld module over $K$ and suppose that there is a trivialisation $\eta$ such that $ {\Pi_{\mr{HTT}}}(\varphi, \eta) \in \mb P^{1}(F_{\gerp})$. Using matrices in $\mr{GL}_2(F_{\gerp})$, we can suppose that the rational line lies in  $T_\gerp(\varphi)$. We proceed as in \cite[Remark III.3.7]{ScholzeTorsion} and we first show that the kernel of 
\[
\mr{HTT}_{\varphi} : T_\gerp(\varphi) \rightarrow \mr{Lie}(\varphi^{D})^{\vee}
\]
is given by $T_\gerp(\varphi^{\mr{CH}})$, where $\varphi^{\mr{CH}}$ is the sub-module of $\varphi$ isomorphic to a power of $\mr{CH}$ (so either $0$ or $\mr{CH}$, by dimension count). This means that we have to show that if $\varphi^{\mr{CH}} =0$, then $\mr{HTT}$ is injective.
Suppose that  $\varphi^{\mr{CH}} =0$, then $\varphi^{D}$ is a formal group, as the \'etale $\pi$-divisible group $F_{\gerp}/A_{\gerp}$ is $\mr{CH}^{D}$.
Given $x \in T_\gerp(\varphi)$, denote again
$x^D : \varphi^{D} \rightarrow \mr{CH}$, and \[\mr{HTT}_{\varphi}(x)= \mr{Lie}(x^D) \in \mr{Lie}(\varphi^{D}).\]
As $\varphi^{D}$ is a formal Drinfeld module, if $\mr{Lie}(x^D)=0$, then $x^D=0$, hence $x=0$.
So if $\mr{Lie}(\varphi) \subset T_\gerp(\varphi)$ then $\mr{Lie}(\varphi) $ is in the kernel of $\mr{HTT}_{\varphi}$ and hence $\varphi^{\mr{CH}} \neq 0$, proving that $\varphi$ is ordinary.
\end{proof}

\begin{rmk}
As in the classical Hodge--Tate decomposition of an abelian variety, the decomposition of $T_\gerp(\varphi)\otimes_{A_{\gerp}} \overline{K}$ is not the one induced by the Hodge decomposition of $\mr{H}^1_{\mr{dR}}(\varphi)$ (see \cite[(3.11)]{GekelerdeRham} or \cite[Lemma 2.21]{HattoriHTT}) and the comparison isomorphism of \cite[Theorem 4.12]{hartl_kim}.
\end{rmk}

\subsection{The perfectoid tower}
For each adic space $h:\mc Y \rightarrow \mc X$ we define $\mc Y(v) := h^{-1}(\mc X(v))$. Recall that the perfection of a ring $R$ in characteristic $p$ is $R^{\mr{perf}}:=\varinjlim R$, where the transition morphisms are given by Frobenius. 
For an adic space $\mc Y = \mr{Spa}(R[1/\pi],R)$, we denote by $\mc Y^{\mr{perf}}$ the adic space associated with $ \mr{Spa}(R^{\mr{perf}}[1/\pi],R^{\mr{perf}})$ (as in \cite[Definition 5.1.1]{KedlayaLiuII}).
Our first goal is to prove the following theorem:
\begin{thm}
Let $v \leq \frac{1}{q^d+1}$. Then there is a perfectoid space 
\[
\mc X_{\infty}(v) \sim \varprojlim_m \mc X(\pi^n)(v)^{\mr{perf}},
\]
where $\sim$ means that the direct limit of functions on the right hand side are dense in the left hand side, and we have an isomorphism of topological spaces 
\[
\vert \mc X_{\infty}(v) \vert \sim \varprojlim_m \vert \mc X(\pi^n)(v) \vert.
\]
\end{thm}

The first step to its proof is as follows:

\begin{thm}
Let $v \leq \frac{1}{q^d+1}$. There is a perfectoid affinoid space 
\[
\mc X_{0,\infty}(v)_a \sim \varprojlim_m \mc X_{0}(\pi^m)(v)^{\mr{perf}}_a,
\]
where $\mc X_{0}(\pi^m)(v)_a$ denotes the anticanonical neighbourhood in $\mc X_{0}(\pi^m)$.
\end{thm}
\begin{proof}
Let \[ h: \mc X(q^{-dm}v) \rightarrow \mc X_0(\pi^{m}).
\]
\noindent
be the map that, on points, sends a Drinfeld module $\varphi$ to $(\varphi/C_{\varphi,m}, \varphi[\pi^m]/C_{\varphi, m})$, where  $C_{\varphi,m}$ is the canonical subgroup of level $m$. By Lemma \ref{lem:lubinkatz} the image of $h$ is contained in  $\mc X_0(\pi^{m}) (v)$. We denote the image of $h$ by $\mc X_{0}(\pi^m)(v)_a$ where `a' stands for anticanonical (as it parametrises Drinfeld modules with a subgroup of the $\pi^m$-torsion that does not intersect the canonical subgroup). We take as an integral model of  $\mc X_{0}(\pi^m)(v)_a$ the integral model $\mf X(q^{-dm}v) $ of $\mc X(q^{-dm}v) $.   As the canonical subgroup is a lift of the kernel of the Frobenius modulo $\pi$ (see the proof \cite[Th\'eor\`eme 4.4]{NicoleRosso}), the projection
\[
\mc X_{0}(\pi^{m+1})(v)_a \rightarrow \mc X_{0}(\pi^m)(v)_a
\]
 coincide with the Frobenius map relative to $A_{\gerp}/\pi$, as the image of  $C_{m+1}$ in $\varphi/C_{\varphi,m}$ is the canonical subgroup). This is a map of degree $q^d$, purely inseparable modulo $\pi$.

If  $\mf X(v)=\mr{Spf}(R)$ and $\mf X(q^{-dm}v) =\mr{Spm}(R_m)$ we have a diagram  
\[
\begin{array}{ccccccc}
R & \rightarrow  & R_1 &              \rightarrow & .... & \rightarrow & R_\infty = \varinjlim R_m \\
\downarrow &    & \downarrow & & & &                                             \downarrow \\
R/\pi & \rightarrow & R/\pi &           \rightarrow & .... & \rightarrow & (R/\pi)^{\mr{perf}}
\end{array}
\]
By the universal properties of perfection, we get a unique map from $R^{\mr{perf}}_\infty$ to $ (R/\pi)^{\mr{perf}}$ which commutes with the last map in the diagram.
We have then 
\[
\mc X_{0,\infty}(v)_a = \mr{Spa}(\widehat{R^{\mr{perf}}_\infty} [1/\pi],\widehat{R^{\mr{perf}}_\infty})
\]
and an isomorphism of topological spaces 
\[
\vert \mc X_{0,\infty}(v)_a \vert \cong \varprojlim_m\vert  \mc X_{0}(\pi^m)(v)_a \vert.
\]
\end{proof}

\begin{rmk}
As perfection commutes with direct limits, it follows that there is a natural map of adic spaces from  $\mc X_{0,\infty}(v)_a$ to $\mc X(v)^{\mr{perf}}$, since perfection does not change the underlying topological space.
\end{rmk}
\begin{lem}
The map 
\[
\mc X(\pi^m)_a \rightarrow \mc X_0(\pi^m)_a
\]
is \'etale.
\end{lem}
\begin{proof}
We just have to check this at the cusps, as the result is known on the open part. From level $\Gamma_0(\pi^n)$ to $\Gamma_1(\pi^n)$ it is a direct calculation on the Tate--Drinfeld module of rank $2$ \cite[Lemma 6.5]{DrinfeldTatevdHein}, that we denote by $\mr{TD}$. This is a rank two Drinfeld module over $A_{\gerp} \llbracket x \rrbracket$ which reduces modulo $x$ to the Carlitz--Hayes Drinfeld module. Note that the variable $x$, contrary to the case of the Tate curve, is not necessarily the uniformiser at the cusp of the Drinfeld modular curve. Anyway, we have that $\mr{TD}[\pi]/\mr{CH}[\pi]$ is generated by an element of positive $x$-adic valuation and it is \'etale as $\pi$-divisible module (cf. the explicit calculation of  \cite[Lemma 6.5]{DrinfeldTatevdHein} or \cite[Lemma 4.4]{HattoriKM}).
Hence the passage from $\Gamma_0(\pi^n)$ to $\Gamma_1(\pi^n)$ is done choosing a generator of this \'etale group, which is unramified. To pass to level $\Gamma(\pi^n)$ we proceed as in \cite[Lemma III.2.35]{ScholzeTorsion}.
\end{proof}

\begin{rmk}
For higher rank $r$, the transition maps from level $\Gamma_1$ to $\Gamma_0$ will not necessarily be \'etale on the boundary, and Scholze uses Tate traces in \cite[III.2.4]{ScholzeTorsion} to deal with this issue. Note that normalized Tate traces are not available in positive characteristic, so it is far from clear to us how to adapt Scholze's strategy for higher ranks.
\end{rmk}

Using this lemma, \cite[Lemma 3.4 (xi)]{BhattScholze}, and Scholze's almost purity result \cite[Theorem 7.9 (iii)]{ScholzeIHES}, we obtain

\begin{thm}
Let $v \leq \frac{1}{q^d+1}$. Then there is an affinoid perfectoid space 
\[
\mc X_{\infty}(v)_a \sim \varprojlim_m \mc X(\pi^m)(v)^{\mr{perf}}_a.
\]
\end{thm}

We now use the Hodge--Tate--Taguchi map to extend the construction to the whole Drinfeld modular curve. We consider the topological space 
\[
\vert \mc X_{\infty} \vert := \varprojlim_m \vert \mc X(\pi^m) \vert.
\]
It parametrises (Tate--)Drinfeld modules $\varphi$ with an isomorphism $A_{\gerp}^2 \rightarrow T_{\gerp} (\varphi)$. There is a natural action of $\mr{GL}_2(A_{\gerp})$ by pre-composition. We can extend this action to an action of $\mr{GL}_2(F_{\gerp})$: given a matrix $ \gamma \in \mr{GL}_2(F_{\gerp})$ with determinant in $A_{\gerp}$  and a Drinfeld module $\varphi$ with a trivialisation $\eta$, we can define a submodule $L:=\eta \circ \gamma (A_{\gerp}^2) \subset T_{\gerp}( \varphi)$, and this corresponds to a subgroup $L_{\mr{coker}}$ of $\varphi[\gerp^m]$, then $L$ is the Tate module of $\varphi/L_{\mr{coker}}$ and $\eta  \circ \gamma $ defines an isomorphism of $L$ with  $A_{\gerp}^2$. We proceed similarly if the determinant has negative valuation.

Let $[x:y]$ be a point in $\mb P^1$, we let $\mr{GL}_2(F_{\gerp})$ act on $\mb P^1$ via 

\begin{align}\label{eq:actionP1}
\left( 
\begin{array}{cc}
a & b \\
c & d 
\end{array}
\right).[x:y]=[dx-by:-cx+ay].
\end{align}

This is simply $\mr{det}(\gamma)\gamma^{-1}$ applied to the vector $ \left( 
\begin{array}{c}
x  \\
y 
\end{array}
\right)$.
If $z=-\frac{y}{x}$ we have the neat formula
\[
\gamma (z) = \frac{az+c}{bz+d}.
\]
\begin{lem}
We have a continuous and $\mr{GL}_2(F_{\gerp})$-equivariant map
\[
\vert {\Pi_{\mr{HTT}}} \vert : \vert \mc X_{\infty} \vert \rightarrow \vert \mb P^1 \vert.
\]
\end{lem}
\begin{proof}
Pointwise the map is defined using the Hodge--Tate--Taguchi period map of Definition \ref{def:HTT}. As $\mr{GL}_2(F_{\gerp})$ acts on the trivialisation of $T_{\gerp} (\varphi)$ in the same way as it acts on $\mb P^1$ and the map is equivariant by the very definition of the action on $\mb P^1$. To prove continuity we note that it is continuous on the inverse image of  $\vert \mc X (v) \vert$ by the existence of the Hodge--Tate--Taguchi map of Theorem \ref{thm:sousgroupe}, which is a map of adic spaces. For the points outside $\vert \mc X (v) \vert$, as in the proof of Theorem \ref{previousthm}, we take the quotient by a suitable subgroup  and this has the effect of dividing the Hasse invariant by $q^d$.
\end{proof}

\begin{lem}\label{lem:openU}
For every $0 < v <1$ there is an open neighbourhood $U$ of $\mb P^1(F_{\gerp})$ such that 
\[
\vert {\Pi_{\mr{HTT}}} \vert^{-1}(U) \subset \vert \mc X_{\infty} (v) \vert.
\]
\end{lem}
\begin{proof}
Same application of quasi-compactness of $ \vert \mc X_{\infty} \vert \setminus  \vert \mc X_{\infty} (v) \vert $ as in  \cite[Lemma III.3.8]{ScholzeTorsion}.
\end{proof}

\begin{lem}\label{lem:cover}
For $U$ as in the lemma above we can find $\gamma_1, \ldots, \gamma_k$ elements of $\mr{GL}_2(F_{\gerp})$ such that 
\[
\vert \mc X_{\infty} \vert = \bigcup_{i=1}^k \gamma_i \vert {\Pi_{\mr{HTT}}} \vert^{-1}(U).
\]
\end{lem}
\begin{proof}
 Let us show that $\mr{GL}_2(F_{\gerp}) U = \mb P^1$. The matrix $\gamma=\left( \begin{array}{cc}
 1 & 0 \\
 0 & \pi
 \end{array}\right)$ sends any point $z=[x:y]\neq \infty$ to $[x:\pi y]$. As $U$ is open and $[x:0]=[1:0] \in \mb P^1(F_{\gerp})$, there is a power of $\gamma$ which send $z$ to $U$. But as $\mb P^1$ is quasi-compact, we can cover it with finitely many translates of $U$ and the fact that $\vert {\Pi_{\mr{HTT}}} \vert$ is $\mr{GL}_2(F_{\gerp})$-equivariant allows us to conclude.
\end{proof}

Summing up these two lemmas we have the following theorem:
\begin{thm}\label{thm:perfectoid}
There is a perfectoid space 
\[
\mc X_{\infty} \sim \varprojlim_m \mc X(\pi^m)^{\mr{perf}}.
\]
\end{thm}
\begin{proof}
By Lemma \ref{lem:cover} we have covered $\vert \mc X(v) \vert $ by finitely many copies of $\vert {\Pi_{\mr{HTT}}} \vert^{-1}(U) \vert$ and by Lemma \ref{lem:openU} this is contained in $\vert \mc X_{\infty} (v) \vert$. So $ \bigcup_{i=1}^k \gamma_i \mc X_{\infty} (v)$ covers $\mc X_{\infty} $. But, by definition, we have that 
\[
\mc X_{\infty} (v) = \mr{GL}_2(A_{\gerp})\mc X_{\infty} (v)_a,
\] 
as at finite level all bases are conjugate to each other. Again by quasi-compactness, we can pick a finite number of matrices $\gamma_j'$ in $\mr{GL}_2(A_{\gerp})$. So $\mc X_{\infty} $ is covered by perfectoid spaces and it is itself a perfectoid space.
\end{proof}
\begin{cor}
There is a $\mr{GL}_2(A_{\gerp})$-equivariant map of adic spaces 
\[
{\Pi_{\mr{HTT}}}: \mc X_{\infty} \rightarrow \mb P^1.
\]
\end{cor}

As in \cite{ScholzeTorsion}, one can show that the whole construction is compatible with changing the tame level, and hence all objects at infinite level admit an action of the prime-to-$\gerp$ Hecke operators, which act trivially on the flag variety.
%%%%%%%%%%%%%%%%%%%%%%%%%%%%%%%%%%%%%%%%%%%%%%%%%%%%%
\subsection{Overconvergent Drinfeld modular forms}
In this section we give a new definition of overconvergent Drinfeld modular forms of weight $s \in \mb Z_p$ \`a la Chojecki--Hansen--Johansson \cite{CHJ}.
In order to proceed, we need the following lemma:
\begin{lem}
Let $\omega=\mr{Lie}(\overline{\mc E})^{\vee}$ be the sheaf of weight one Drinfeld modular forms on $ \mc X$ and $\omega_{\infty}$ the pullback to $\mc X_{\infty}$. Over $\mc X_{\infty}(v)$, we have 
\[
\omega_{\infty}  = \Pi_{\mr{HTT}}^* \mc O(1).
 \]
\end{lem}
\begin{proof}
If we interpret $\mb P^1$ as the flag variety for lines in a $2$-dimensional vector space,  then we have locally that $\mr{Lie}(\overline{\mc E})$ identifies naturally with the structural sheaf $\mc O$. Hence $\omega_{\infty} $ identifies with the tautological line bundle. 
\end{proof}
We define $\mb P^1_v=\set{z \in \mb P^1 \vert \exists z_0 \in \pi A_{\gerp} \mbox{ s.t. }|z-z_0|< q^{-dv}<1}$. This set is stable for the action of $\Gamma_0(\gerp)$. We also define $\mc U_{\infty,v}$ to be $\Pi_{\mr{HTT}}^{-1}(\mb P^1_v)$. 

\begin{lem}\label{lem:neighOrdlocus}
Let $q_{\infty}: \mc X_{\infty} \rightarrow \mc X_0(\pi)$ be the natural projection and $q_{\infty}^{\mr{perf}}: \mc X_{\infty} \rightarrow \mc X^{\mr{perf}}_0(\pi)$. Then the $\mc X_v := q_{\infty}(\mc U_{\infty,v})$, with $v$ tending to $0$, are a set of strict affinoid neighbourhoods of the ordinary multiplicative locus in $\mc X_0(\pi)$, and similarly for $\mc X^{\mr{perf}}_v := q^{\mr{perf}}_{\infty}(\mc U_{\infty,v})$.
\end{lem}
\begin{proof}
First note that $q_{\infty}(\mc U_{\infty,v})$ is open as $q_{\infty}$ is pro-\'etale.
We have to show that a point $(\varphi, \eta) \in \mc U_{\infty,v}$ defines a Drinfeld module in $\mc X_{v'}$ for some $v'$, and that $\eta$ restricted to the first coordinates trivialises the canonical subgroup of $\varphi$. By Lemma \ref{lem:P^1(F)}, we know that  the image of the ordinary locus is contained in $\mb P^1_v$. This implies that we can suppose that $\varphi$ has good reduction and  that $\varphi \in \mc X_{v'}$, for a suitable $v'$. Using the explicit form of the canonical subgroup \cite[Th\'eor\`eme 4.4]{NicoleRosso}, we see that if $v$ tends to $0$ then $v'$ also tends to $0$ (N.B. $v'$ is denoted $w$ in {\it loc. cit.}). 
It is also clear from the definition of the Hodge--Tate--Taguchi period map that the defining line in the ordinary case is given by the canonical subgroup. As $q_{\infty}(\mc U_{\infty,v})$ is open, this implies that $\eta$ restricted to the first coordinate must give a point in the neighbourhood of the multiplicative ordinary locus.
To check that $q_{\infty}(\mc U_{\infty,v})$ is affine we note that the pre-images by $\Pi_{\mr{HTT}}$ of the standard cover of $\mb P^1$ are affinoids, and  hence the inverse image of the affinoid $\mb P^{1}_v$ is affinoid. By the definition of the topology in the inverse limit
\[
\vert \mc X_{\infty}(v) \vert \cong \vert \varprojlim_m \mc X(\pi^n)(v)\vert,
\]
affinoids always come from a finite level. One can conclude by taking invariants thanks to \cite[Corollary 6.26]{CHJ}.
\end{proof}

Let $\mf z$ be the pullback via the Hodge--Tate--Taguchi map of the coordinate $z$ on $\mb P^1$. In \cite{ScholzeTorsion} this is called a fake Hasse invariant, as it commutes with prime-to-$\gerp$ Hecke operators.
Note that  $\omega_{\infty \vert_{\mb P^1_v}} = \mc O_{\mc U_{\infty,v}}$; indeed we can define an element $s$ which trivialises $\mc O(1)$ as in \cite[\S 2.4]{CHJ}. 
Identify $\mc O(1)$ with the contracted product $(\mr{GL}_2 \times \mb A^1)/ B$ where $B$ is the Borel subgroup of lower triangular matrices and $\gamma$ in $B$ acts on $a \in \mb A^1$ via multiplication by $d^{-1}$ and on $\mr{GL}_2$ by right multiplication. A global section is hence a map $f: \mr{GL}_2 \rightarrow \mb A^1$ such that $f(gh)=d_{h}f(\gamma)$, for $g$ in $\mr{GL}_2$, $h$ in $B$ and $d_{h}$ is the right lower entry of $h$. The function $s$ sending $g$ to $-b_{g}$ i.e., minus the upper right entry  of $g$, satisfies this condition. Moreover $s=0$ if and only if $g$ is in $B$. Hence $s$ is a non-vanishing section of $\mc O(1)$ on $\mb P^1 \backslash \set{\infty}$.
We want to see how $\gamma$ in $\mr{GL}_{2}(F_{\gerp})$ acts on $s$; recall the action on $\mb P^1$ given in \eqref{eq:actionP1}. If $g$ corresponds to $[x:y]$, then $\gamma. [x:y]$ is the image of  $\mr{det}(\gamma)\gamma^{-1}g$ in $(\mr{GL}_2/B)(\mr F_{\gerp})$. Let $g= \left ( 
\begin{array}{cc}
0 & -1 \\
1 & z
\end{array} \right)$, then 
\[
\gamma^* s (g) = s(\mr{det}(\gamma)\gamma^{-1}g) = (bz+d)s(g).
\]
Then, let $\mf s$ be the pullback via the HTT map of $s$, which trivialises $\omega_{\infty \vert_{\mb P^1_v}}$. We then have a cocycle 
\[
j(\gamma, \mf z)= \frac{\gamma^* \mf s}{ \mf s}= (b \mf z +d).
\]

\begin{dfn} Let $v \leq \frac{1}{q^d+1}$. The space of perfectoid overconvergent Drinfeld modular forms of weight $s$ and radius of overconvergence $v$ is 
\[
M_s(v):=\set{f : \mc U_{\infty,v} \rightarrow C \vert f(\gamma \mf z )= {(b\mf z+d)}^{-s}f(\mf z), \forall \gamma \in \Gamma_0(\gerp)}.
\]
\end{dfn}

We want to compare this definition to that in \cite[D\'efinition 4.11]{NicoleRosso}. We define a sheaf $\omega^{s}_v$ on $q^{\mr{perf}}_{\infty}(\mc U_{\infty,v})$ as
\[
\omega^{s}_v (\mc U)= \set{f :{q^{\mr{perf}}_{\infty}}^{-1}( \mc U) \rightarrow C \vert f(\gamma\mf z )= {(b\mf z+d)}^{-s}f(\mf z), \forall \gamma \in \Gamma_0(\gerp)}.
\]
We shall show that this sheaf is a line bundle on $q^{\mr{perf}}_{\infty}(\mc U_{\infty,v})$ under assumption \ref{assum:proetale}.  This amounts to finding locally a generator for it.

As in \cite[Definition 4.1]{ScholzeHodge} we give the following definition:

\begin{dfn}
Let $\mc X$ be a rigid analytic variety or the perfection of a rigid analytic variety over $\mr{Spa}(F_{\gerp},A_{\gerp})$. On $\mc X_{\mr{proet}}$ we define the the integral completed structure sheaf 
$$ \hat{ \mc O}_{\mc X}^+ = \varprojlim { \mc O}_{\mc X}^+/\pi^n  $$ and the completed structure sheaf
\[ \hat{ \mc O}_{\mc X}= \hat{ \mc O}_{\mc X}^+ \left[ \frac{1}{\pi} \right]. \]
\end{dfn}

We make the following
\begin{ass}\label{assum:proetale}
Let $K  \subset \mr{GL}_2(\mb Z_p)$ and $\mc X (K)$ the corresponding rigid space. We suppose that for all admissible opens $\mc U$ of $q^{\mr{perf}}_{\infty}(\mc U_{\infty,v})$ in the pro-\'etale site $\mc X (K)^{\mr{perf}}_{\mr{proet}}$, $\hat{ \mc O}^+_{\mc X(K)^{\mr{perf}}}(\mc U)$ is the $\pi$-adic completion of ${ \mc O^+}_{\mc X(K)^{\mr{perf}}}(\mc U)$.
\end{ass}

\noindent
As pointed out in \cite{ScholzeHodge} it is not obvious in general that $ \mc O^+_{\mc Y}(\mc U)$ is dense in $\hat{ \mc O}^+_{\mc Y}(\mc U)$ but nonetheless this is known for rigid varieties over $\mb Q_p$ thanks to \cite[Corollary 6.19]{ScholzeHodge}.

\begin{lem}\label{lem:invariants}
For every level $K \subset \mr{GL}_2(\mb Z_p)$ we have 
\[
\mc O_{\mc X_{\infty}}^{K}= \mc O_{\mc X(K)^{\mr{perf}}},
\] 
which means $\mc O_{\mc X_{\infty}}(q_{\infty}^{-1}(\mc U))^{K} = \mc O_{\mc X(K)^{\mr{perf}}}(\mc U)$ for every admissible open $\mc U$.
\end{lem}
\begin{proof}
We can check it locally. First suppose that $\mc U$ does not intersect the boundary. Applying \cite[Lemma 2.24]{CHJ} to $\mc O^+_{\mc X_{\infty}}/ \pi^m(q_{\infty}^{-1}(\mc U))$, we get that  
\[
\hat{\mc O}^+_{\mc X_{\infty}}(q_{\infty}^{-1}(\mc U))^{K}=\hat{ \mc O}^+_{\mc X(K)}(\mc U),
\] 

where $\hat{\phantom{e}}$  denotes the integral completed structure sheaf on the pro-\'etale site.
We have the following inclusions:
$$
\begin{array}{ccccc}
{ \mc O^+}_{\mc X(K)^{\mr{perf}}}(\mc U)&\subset & \hat{ \mc O}^+_{\mc X(K)^{\mr{perf}}}(\mc U) &=& \hat{\mc O}^+_{\mc X_{\infty}}(q_{\infty}^{-1}(\mc U))^{K}\\
{ \mc O^+}_{\mc X(K)^{\mr{perf}}}(\mc U)& \subset & {\mc O^+}_{\mc X_{\infty}}(q_{\infty}^{-1}(\mc U))^{K} &=& \hat{\mc O}^+_{\mc X_{\infty}}(q_{\infty}^{-1}(\mc U))^{K}.\\
\end{array}
$$
If Hypothesis \ref{assum:proetale} holds, then  $\hat{ \mc O}^+_{\mc X(K)^{\mr{perf}}}(\mc U)$ is the $\pi$-adic completion of $\set{ \mc O^+_{\mc X(K)^{\mr{perf}}} (\mc U)/\pi^n}$ and hence the objects in the first line are all the same, and this forces the first inclusion in the second line to be an equality.

If $\mc U$ intersects the boundary, any section $f$ of $\mc O_{\mc X_{\infty}}(q_{\infty}^{-1}(\mc U))^{K}$ defines a bounded section of $\mc O_{\mc X_{\infty}}(q_{\infty}^{-1}(\mc U) \setminus \set{ \mr{cusps}})^{K}$ and hence, by the previous case, a bounded element of ${ \mc O}_{\mc X(K)^{\mr{perf}}}(\mc U  \setminus \set{ \mr{cusps}})$. Write $f$ as the $\pi$-adic limit $\varinjlim f_n$, where every $f_n$ comes from pullback from a finite level in the limit $\mc X(K)^{\mr{perf} } = \varprojlim \mc X(K)$. Each $f_n$ can be extended uniquely to an element $\tilde{f}_n$ of $ \mc O_{\mc X(K)}(\mc U)$ by \cite[Theorem 1.6 I)]{LutkExtensionCodim}. We want to show that $\tilde{f}= \varinjlim \tilde{f}_n$ is well-defined. It is enough to check that they form a Cauchy sequence for the $\pi$-adic topology. If $\tilde{f}_n $ and $\tilde{f}_m$ are defined on the same $\mc X(K)$ and $\mr{sup}\vert f_n -f_m\vert_{\pi}< \epsilon$ on $\mc U \setminus \set{\mr{cusps}}$, the extension $\tilde{f}_n -\tilde{f}_m$ has sup-norm $\leq \epsilon$ by the maximum modulus principle, hence $\tilde{f}$ is a well-defined element of $\mc O_{\mc X(K)^{\mr{perf}}}(\mc U)$.
\end{proof}

Let $\xi_{\mc U}$ be a basis of $\omega^{\mr{perf}}$ over $q^{\mr{perf}}_{\infty}(\mc U)$, and let $\mf s = t_{\mc U}  {q^{\mr{perf}}_{\infty}}^{-1}(\xi_{\mc U})$, for $t_{\mc U}$ a unit in $\mc O_{\mc X_{\infty}}(\mc U)$. We want to make sense of $ t_{\mc U}^s$ for $s \in \mb Z_p$. As in \cite[Proposition 2.27 ]{CHJ} we have:
\begin{lem}\label{lem:factor}
We can write 
\[
t_{\mc U}= t'_{\mc U} s'_{\mc U},
\]
with $t'_{\mc U} \in 1+\pi\mc O_{\mc X_{\infty}}(\mc U)$ and $s'_{\mc U} \in \mc O_{\mc X_0(\gerp)^{\mr{perf}}}$. Moreover $\gamma^* t_{\mc U} = (b \mf z + d)t_{\mc U}$.
\end{lem}
Let $f$ in $\omega^{s}_v (\mc U)$ and note that 
\[
\gamma^*(f  {t'_{\mc U}}^s) = f  {t'_{\mc U}}^s \left(\frac{s'_{\mc U}}{\gamma^*s'_{\mc U}}\right)^s.
\]
As $s'_{\mc U}$ is  invariant by $\Gamma_0(\gerp)$, applying lemma \ref{lem:invariants} to  $f  {t'_{\mc U}}^s$ we can then embed 
\[
\omega^{s}_v (\mc U) \rightarrow \mc O_{\mc X(\gerp)^{\mr{perf}}}(\mc U)
\]
via $f \mapsto f  {t'_{\mc U}}^s$.
\begin{thm}
The sheaf $\omega^{s}_v$ is coherent and locally free of rank one.
\end{thm}
\begin{proof}
For $s=0$ this is consequence of Lemma \ref{lem:invariants}. Note that we can make $\Gamma_0(\gerp)$ acts on $\mc O_{\mc X_0(\gerp)^{\mr{perf}}}(\mc U)$ via
\[
\gamma.f=\left(\frac{\gamma^*s'_{\mc U}}{s'_{\mc U}}\right)^s \gamma^* f.
\]
Then $\omega^{s}_v (\mc U)$ falls in the invariant part for this action. 

Note that if $\mc U$ does not intersect the boundary, the cover is \'etale Galois and by Galois descent the invariant part is locally free on $\mc U$ of rank one (the rank of  $\mc O_{\mc X_0(\gerp)}$). So $\omega^{s}_v$ is a subsheaf of a rank one sheaf. It is indeed coherent: take a section $f$ and suppose it is not vanishing (shrinking $\mc U$ if necessary). Division by $f$ induces an isomorphism between $\omega^{s}_v$ and $\omega^{0}_v$, which is coherent.
We can extend everything to the boundary reasoning as in the last part of the proof of Lemma \ref{lem:invariants}.
\end{proof}

We are almost ready to compare this notion of overconvergent Drinfeld modular forms with our previous notion, but first we need to compare line bundles on $\mc X_0(\gerp)^{\mr{perf}}$ and $\mc X_0(\gerp)$. 
If $X$ is a scheme, it is known that 

$$ \mr{Pic}(X)\left[ \frac{1}{p}\right]= \mr{Pic}(X^{\mr{perf}}),$$
{\it i.e. } every line bundle comes from finite level by pullback of line bundles on the Frobenius twist $X^{(p^n)}$ \cite[Lemma 3.5]{BhattScholze}. The same holds for rigid spaces:

\begin{prop}
Let $\mc Y$ be a rigid analytic space of $F_{\gerp}$ and $\mc Y^{\mr{perf}}$ its perfection. Then 
$$ \mr{Pic}(\mc Y)\left[ \frac{1}{p}\right]= \mr{Pic}(\mc Y^{\mr{perf}}).$$
\end{prop}
\begin{proof}
This is a direct consequence of \cite[Lemma 5.6.8]{KedlayaLiuII} which says that every finite projective module on the perfection comes via extension of scalars from a finite level. 
\end{proof}

Recall the rigid torsor $\mc F$ of \cite[D\'efinition 4.5]{NicoleRosso}: it parametrises generators of the image of the Hodge--Tate--Taguchi map in $\omega$. It is a torsor for $G:={ A_{\gerp}^{\times} \Big(1 + \pi^v \mc O_{\mc X(v)} \Big)} $. 

We pull it back to a sheaf $\mc F^{\mr{perf}}$ on $q^{\mr{perf}}_{\infty}(\mc U) \rightarrow \mc X_0(\gerp)^{\mr{perf}}$ and it is now a torsor for $G^{\mr{perf}}={ (A_{\gerp}^{\times})^{1/p^{\infty}} \Big(1 + \pi^v \mc O_{\mc X(v)^{\mr{perf}}} \Big)}$. Here $(A_{\gerp}^{\times})^{1/p^{\infty}}$ is the group of all $p^{n}$-roots of elements of $A_{\gerp}^{\times}$ which is isomorphic to a finite product of infinitely many copies of $\mb Q_p$. Using the arguments of the proof of Lemma \ref{lem:neighOrdlocus} on the fact that $X_v$ is a neighbourhood of the ordinary locus  we get:
\begin{lem}
The section $s'_{\mc U} {q^{\mr{perf}}_{\infty}}^{-1}(\xi_{\mc U})$ of Lemma  \ref{lem:factor} is a generator of $\mc F^{\mr{perf}}$. 
\end{lem}

\noindent By abuse of notation, we denote by the same symbol the corresponding element of $G^{\mr{perf}} $ (which can be seen as an inverse limit of generators of the images of the Hodge--Tate--Taguchi map in $\omega$ along the Frobenius tower).

We can go from the new sheaf $\omega^s_v$ to the perfection of the older sheaf $\omega^{s,\mr{NR}}$ as follows. First note that there is no action of $G$ on the sections of $\omega^s_v$. So given $f$ in $\omega^{s,\mr{NR}}(\mc U)$ the function $(s'_{\mc U} {q^{\mr{perf}}_{\infty}}^{-1}(\xi_{\mc U}))^s  f $ is an element of $\omega^s_v(\mc U)$ which transforms correctly for the action of $\Gamma_0(\gerp)$ by the cocycle relation and it is invariant by the action of $G$ as $f$ is homogenous of weight $-s$. This map is bijective, hence:

\begin{thm}
The sheaf $\omega^s_v$ is the pullback of  $\omega^{s,\mr{NR}}$ on the perfection $\mc X_v^{\mr{perf}}$ .
\end{thm}
% Summing up,
%\begin{thm}\label{thm:comparison}
%We have a Hecke equivariant isomorphism  
%\[
%\varinjlim_v M_s(v)  =\widehat{(\varinjlim_{v'}  (\varinjlim_{\mr{Frob}} \mc M_{s,v'})}^{\pi}.
%\]
%where $\varinjlim_{v'} \mc M_{s,v'}$ is as in \cite[Th\'eor\`eme 4.21]{NicoleRosso}.
%\end{thm}

%Note that $\mc M_{s,v'} \stackrel{\mr{Frob}}{\rightarrow} \mc M_{s,v'}$ sends a form of weight $s$ to a form of weight $q^d s$ for the action of the group $G^{1/p}$. 
%%%%%%%%%%%%%%%%%%%%%%%%%%%%%%%%%%%%%
\section{Some open problems on families of Drinfeld modular forms}

\subsection{Modularity theorem}

In the appendix of \cite{NicoleRosso}, we showed that the space of ordinary modular forms of level $T$ and rank $2$ is one-dimensional over the Iwasawa algebra, proving an $r=t$ theorem. The restriction on working only with level $T$ is imposed by the fact that the ramification of the Galois representation associated with a Drinfeld modular forms is not known. More precisely, given a $\gerp$-adic Galois representation, we do not know if the ramification at a different prime $\gerq$ is finite or not.  If we could develop a theory of vanishing cycles for B\"ockle--Pink $\tau$-sheaves \cite{BocklePink}, it is likely that we could show that the ramification at $\gerq$ is finite, and that a similar $r=t$ theorem could be proven. 

We present a conjectural application which has been suggested to us by C. Popescu. Let $\chi$ be an element of $\mr{Pic}^0(A)$ and $\varphi_{\chi}$ the corresponding Carlitz--Hayes module. It is known that the Galois representation $\rho_{\varphi_{\chi}}$ is unramified outside $\gerp$ and $\infty$ (see \cite[Theorem 5]{Takahashi_ellmod} and \cite[Theorem 3.2]{Hayes}).  An $r=t$ theorem  would show that $\rho_{\varphi_{\chi}}$ is `modular' of type II \cite[Definition 6]{GossModularity}, i.e., the Galois representation on the $\gerp$-torsion of $\varphi_{\chi}$ arises from a $\pi$-adic Drinfeld modular form. This is only known for the Carlitz module of $\mb F_q[T]$ \cite[Example 10]{GossModularity} which is associated to the  Drinfeld modular form $\Delta$, which is not ordinary for any prime.

\subsection{Non noetherian eigenvarieties}

In his eigenvariety paper, Buzzard points out that Lemma A1.6 of \cite{Col} is not complete: roughly speaking, Coleman claims that given a completely continuous operator $\mr U$ on an orthonormalisable module $M$ over a Banach algebra $A$ and a finitely generated submodule $M'$, a finite number of coordinates are enough to determine if an element of $M'$ is $0$ or not. In the noetherian case, this is handled by \cite[Lemma 2.3]{Buz}. 
It would interesting to investigate the following:

\begin{question}
Given a completely continuous operator $\mr U$ on an orthonormalisable module $M$ over a non-noetherian Banach algebra $A$, can one define a Fredholm determinant $F_{\mr U}(X)$? If not, are there extra conditions on $A$ under which this holds true?
\end{question}

Once we have a good definition of the Fredholm determinant $F_{\mr U}(X)$, the next step is the construction of the spectral variety, which is defined as the closed subset \[ \mc Z:= V(F_{\mr U}(X)) \subset \mb A^1_{\mr{Spa}(A,A^{\circ})}.\] 
A difficult point is the generalisation of \cite[Lemma 4.1]{Buz} which proves that $\mc Z$ is flat, as for the moment we lack a flatness criterion for non-noetherian rings. For example, we believe it would be enough to generalise \cite[Lemma 10.127.4]{StacksProject}, replacing `essentially of finite presentation' by `topologically of finite type'.

\subsection{Families of modular forms for $t$-motives}
An obstacle to generalising Drinfeld modular forms in higher dimensions is our lack of understanding of algebraic families of Anderson $A$-motives. Still, the local theory as developed in \cite{hartl_singh} gives a nice description of the duality between local Anderson modules and local shtukas. The properties of the Hodge--Tate--Taguchi map in higher dimension are not as tractable as in dimension one: indeed, the proof of the almost surjectivity in Theorem \ref{thm:sousgroupe}, in our one-dimensional case of arbitrary rank, was done by hand. For abelian varieties (or more generally, $p$-divisible groups), the proof of almost surjectivity is done using $p$-adic Hodge theory (see \cite[Appendix C]{FarguesIsom} or \cite[\S III.2.1]{ScholzeTorsion}), whose analogue is lacking in our context.

% Is it possible to control its cokernel? 

\subsection{The maximal slope}

Given an elliptic modular form over $\mathbb{Q}$ of weight $k$, we know that if $p$ does not divide the level, the possible $\mr U_p$-eigenvalues have slopes between $0$ and $k-1$. This is because the constant term of the Hecke polynomial is, up to a root of unit, $p^{k-1}$.

For Drinfeld modular forms, the Hecke polynomial has constant term $0$, which a priori allows any possible slope.
\\

In level $\Gamma_0(T)$, Bandini and Valentino conjecture that the maximal slope is always $(k-2)/2$, and this maximal slope arises exclusively from (suitably defined) newforms. Recently, they proved a quadratic bound for the maximal slope \cite[Theorem 6.4]{BandiniValentino3}.

In level $\Gamma_1(T)$, explicit computations for $A = \mb F_q [T]$ hint to the fact that the maximal slope is at least always bounded by $k-1$, if not better estimates.
The following Figures 1--4 are calculated using Hattori's tables \cite{HattoriTable} relying on the formulae of Bandini and Valentino \cite{BandiniValentino,BandiniValentino2}. We use the $x$-axis to indicate the weight and the $y$-axis to indicate the maximal slope appearing in that fixed weight. Note that the patterns of the maximal slope distribution vary widely with $q$. At the time of writing, we have no clue towards a conceptual explanation of this variation.

\begin{figure}
  \includegraphics[width=\linewidth]{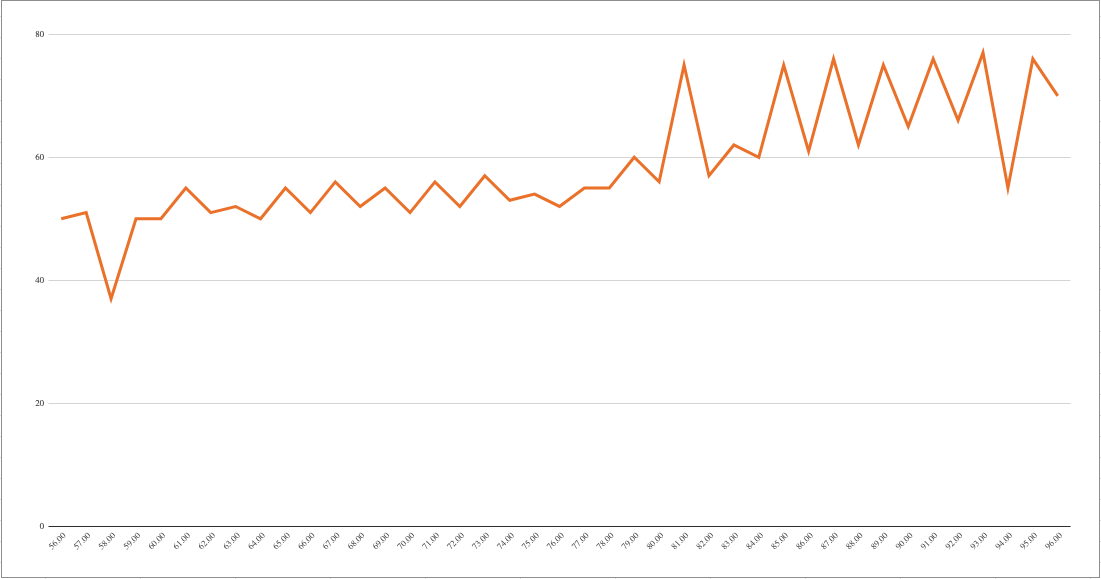}
  \caption{$q=5$}
  \label{fig:5}
\end{figure}

\begin{figure}
  \includegraphics[width=\linewidth]{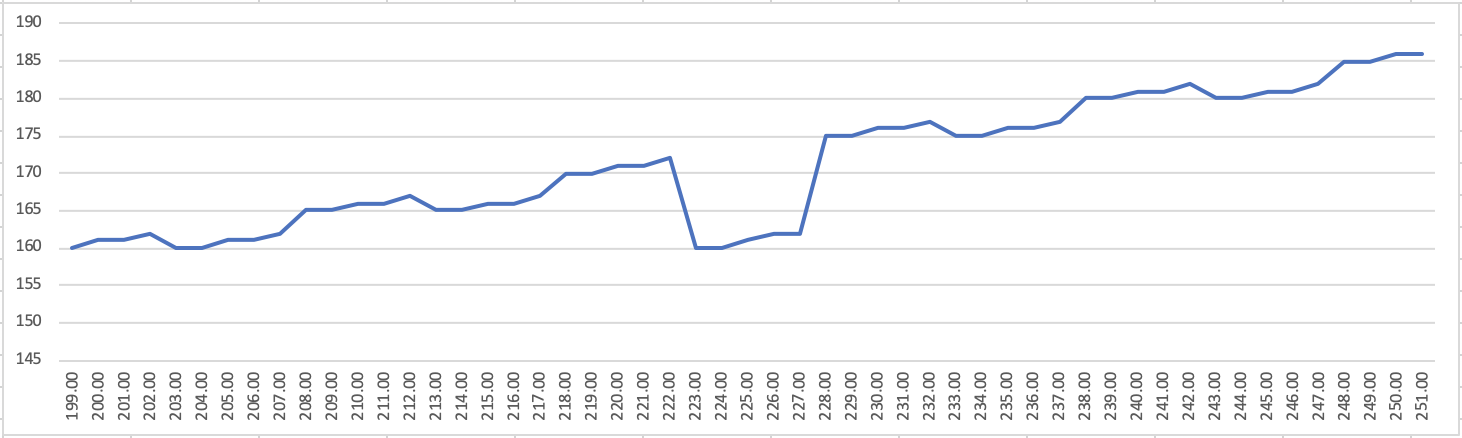}
  \caption{$q=125$}
  \label{fig:125}
\end{figure}

\begin{figure}
  \includegraphics[width=\linewidth]{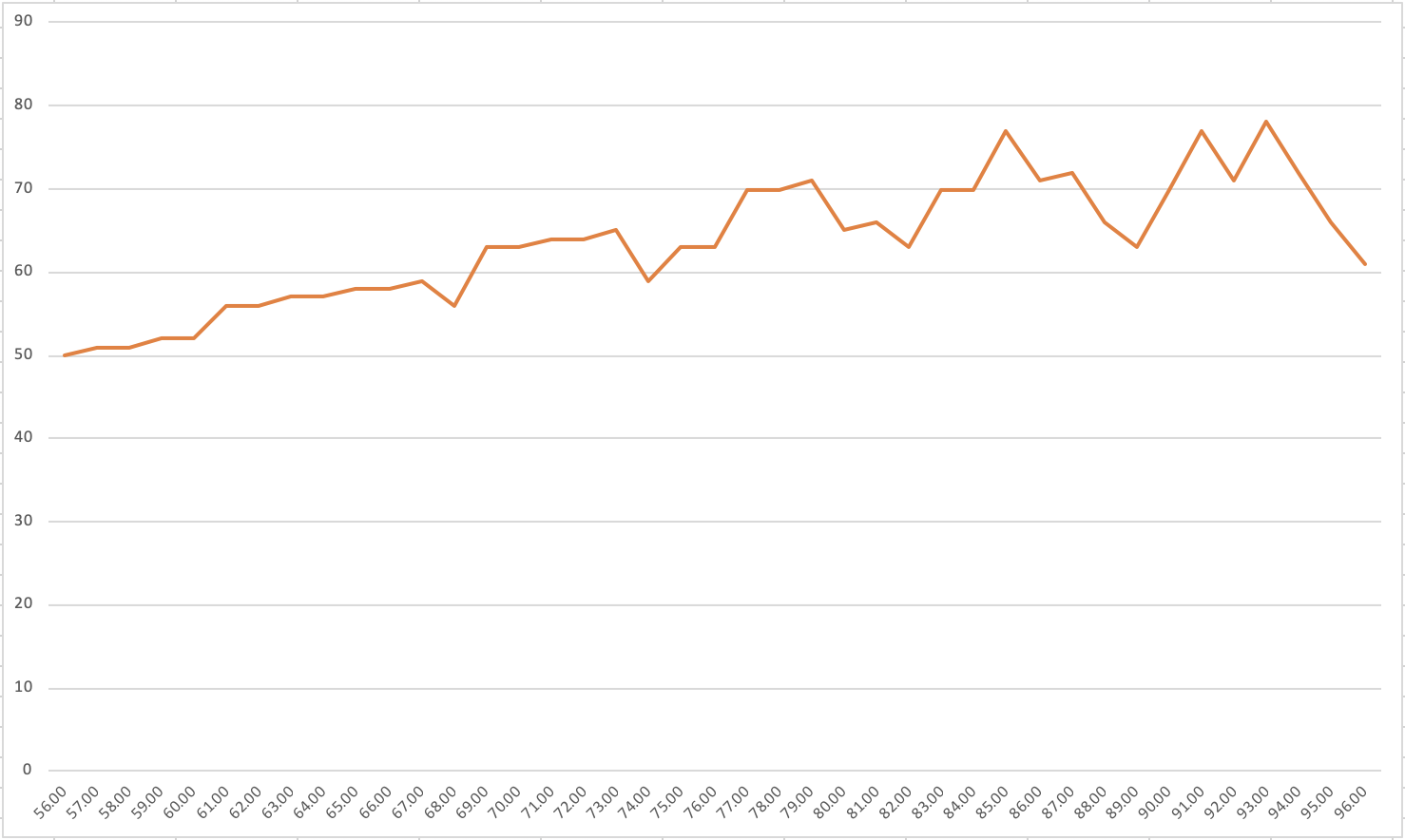}
  \caption{$q=7$}
  \label{fig:7}
\end{figure}

\begin{figure}
  \includegraphics[width=\linewidth]{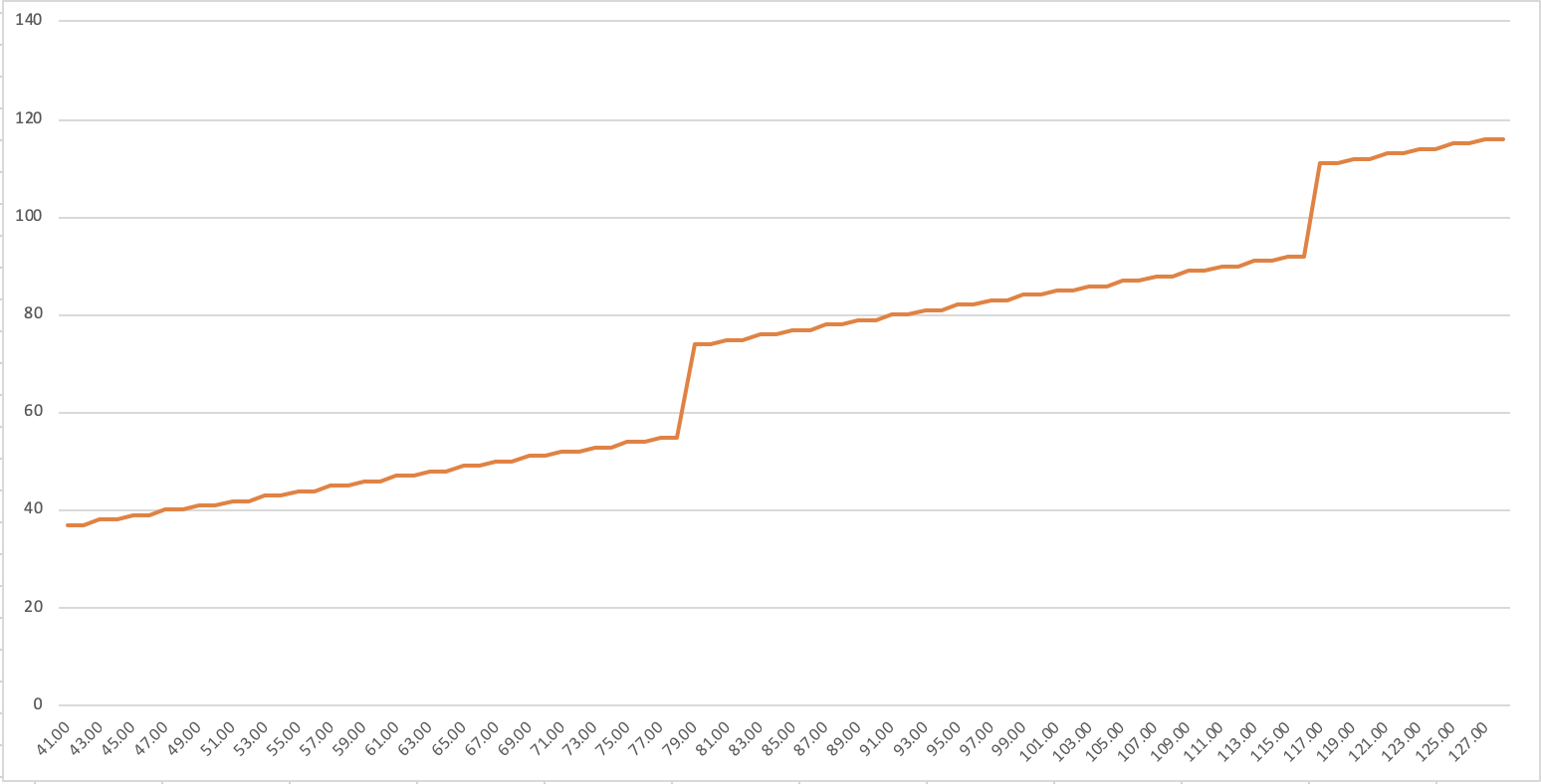}
  \caption{$q=37$}
  \label{fig:37}
\end{figure}

\subsection{Infinite slope  Drinfeld modular forms}

%\subsubsection{Examples of classical DMF of infinite slope}
Suppose that $\gerp$ is principal and let $f$ be an eigenform of $\mr T_{\pi}$ of prime-to-$\pi$ level and eigenvalue $\lambda$. In rank two, the explicit formula for the $\pi$-stabilisation of $f$ \cite[\S 3.2]{BandiniValentino} (with $m=0$ and a slightly different normalisation of $\mr U_{\pi}$, as ours is theirs divided by $\pi$) tells us that
\[ \mr U_{\pi} \left( f(z) - \frac{\pi^{k-1}}{\lambda} f( \pi z) \right)= \lambda \left( f(z) - \frac{\pi^{k-1}}{\lambda} f( \pi z) \right), \;\; \mr U_{\pi}  f(\pi z)=0 . \]

\noindent
This means that we have plenty of modular forms of level $\Gamma_0(\pi)$ and infinite slope, which is never the case for classical modular forms! In \cite[Corollaire 5.10]{NicoleRosso}, we show that if $f$ is an overconvergent modular form of weight $k$ and slope smaller than $k-1$, then it is classical of level $\Gamma_0(\pi)$. A natural question is then the following:

\begin{que}
Given an overconvergent Drinfeld modular form of weight $k$ and infinite slope, is there a criterion to decide whether $f$ is classical?
\end{que}

\noindent
Moreover, at least in the ordinary case, we know from \cite[Th\'eor\`eme 3.14]{NicoleRosso} that if the weight is large enough, the form $f$ is not only classical, but it comes via $\pi$-stabilisation from a form of prime-to-$\pi$ level. Hence, for a Zariski dense set of points $\set{f_k}$ in an ordinary family, we can find a corresponding classical Drinfeld modular form $\tilde{f_k}$ of infinite slope, whose prime-to-$\pi$ Hecke eigenvalues vary in a Iwasawa algebra. This leads naturally to the following:

\begin{que} Do continuous families of infinite slope Drinfeld modular forms exist? 
\end{que}

\subsection{Horizontal control theorem}

An alternative approach to classical Hida theory does not vary the weight of the modular forms but varies instead the level at $p$ of the modular curve, and then shows that the ordinary parts of the $\mr{H}^1$ of these curves glue to a finitely generated $\Lambda$-adic module. This approach seems much harder to understand for function fields: for example, there are only $q^d-1$ finite order characters of $\left( A_{\gerp}/\gerp^r\right)^{\times} $, independently of $r$. 

As far as we know, the best known result towards a horizontal control theorem  la Hida as alluded to above is Marigonda's unpublished 2008 PhD thesis for Drinfeld modular curves \cite[Theorem 11]{Marigonda}.

\begin{thm}
 Let $J_n$ be the Jacobian of $\mc X_1(\pi^n)$ and $G_r = \left( 1 + \pi A_{\gerp}\right)/\left( 1 + \pi^r A_{\gerp}\right)$. Then the ordinary part of the $p$-adic Tate module $T_p(J_n)$ is free over $\mb Z [G_r]$ of rank bounded by the rank of $T_p(J_n)$.
\end{thm}

\noindent
As the rank of the latter is known to grow with the index of $\Gamma_1(\pi^r)$, this is not telling us much on the possibility of a horizontal Hida family. Moreover, the alternative approach to Hida theory due to Emerton \cite{EmertonHida} does not seem to apply here, as it relies on the fact that $\mb Z_p^{\times}$ is topologically generated by one element, while we are very much in a non-noetherian situation.

\begin{que} 
Is the inverse limit $
\varprojlim_n T_p(J_n)^{\mr{ord}} $
 a finite free $\Lambda$-module?
\end{que}

\bibliographystyle{amsalpha}
\bibliography{referencesDrinfeld}
\end{document}